\documentclass[12pt,a4paper]{article}
\usepackage{amsmath,amsfonts,amssymb,amsthm,amscd,verbatim}

\def\i{\iota}

\def\u{\upsilon}

\chardef\tempcat=\the\catcode`\@ \catcode`\@=11
\def\cyracc{\def\u##1{\if \i##1\accent"24 i
    \else \accent"24 ##1\fi }}
\newfam\cyrfam

\DeclareFontFamily{OT1}{msb}{}{} \DeclareFontShape{OT1}{msb}{m}{n}
 {  <5> <6> <7> <8> <9> <10> gen * msbm
      <10.95><12><14.4><17.28><20.74><24.88>msbm10}{}
\DeclareMathAlphabet{\bubble}{OT1}{msb}{m}{n}

\newfont{\goth}{eufm10 scaled \magstep1}

\def\gg{\mbox{\goth g}}
\def\gh{{\mbox{\goth h}}}

\def\gm{\mbox{\goth m}}
\def\gn{\mbox{\goth n}}

\def\gq{\mbox{\goth q}}

\newfont{\mcal}{eusm10 scaled \magstep1}

\def\square{\kern1pt\vbox
            {\hrule height 0.6pt\hbox{\vrule width 0.6pt\hskip 3pt
 \vbox{\vskip 6pt}\hskip 3pt\vrule width 0.6pt}\hrule height 0.6pt}\kern1pt}

\def\ra{\to}

\newtheorem{Th}{Theorem}
\newtheorem{Prop}{Proposition}
\newtheorem{Cor}{Corollary}
\newtheorem{Lem}{Lemma}
\newtheorem{Def}{Definition}
\newtheorem{Ex}{Exercise }
\newtheorem{Exa}{Example}
\newtheorem{rem}{Remark}

\def\bexa{\begin{Exa}}
\def\eexa{\end{Exa}}

\def\bt{\begin{Th}}
\def\et{\end{Th}}
\def\bp{\begin{Prop}}
\def\ep{\end{Prop}}
\def\bc{\begin{Cor}}
\def\ec{\end{Cor}}
\def\bl{\begin{Lem}}
\def\el{\end{Lem}}
\def\bd{\begin{Def}}
\def\ed{\end{Def}}
\def\bex{\begin{Ex}}
\def\eex{\end{Ex}}

\def\be{\begin{equation}}
\def\ee{\end{equation}}
\def\ben{\begin{enumerate}}
 \def\een{\end{enumerate}}
\def\ba{\begin{array}{rlll}}
\def\ea{\end{array}}
\def\bea{\begin{eqnarray}}
\def\eea{\end{eqnarray}}
\def\bean{\begin{eqnarray*}}
\def\eean{\end{eqnarray*}}

\catcode`@=11 \@addtoreset{equation}{section} \catcode`@=12

\begin{document}

\title{Tanaka  structures (non holonomic $G$-structures) and Cartan connections}

\author{Dmitri V. Alekseevsky and Liana David}

\maketitle

{\bf Abstract} Let $\gh = \gh_{-k} \oplus \cdots \oplus \gh_{l}$
($k> 0$, $l\geq 0$) be a finite dimensional graded Lie algebra,
with a Euclidean  metric $\langle \cdot , \cdot\rangle$ adapted to the
gradation. The metric $\langle\cdot ,\cdot \rangle$ is called admissible
if the codifferentials $\partial^{*} : C^{k+1} (\gh_{-}, \gh ) \ra
C^{k} (\gh_{-} , \gh )$ ($k\geq 0$) are
$Q$-invariant ($\mathrm{Lie}(Q) = \gh_{0}\oplus
\gh_{+}$). We find necessary and sufficient conditions for a
Euclidean metric, adapted to the gradation, to
be admissible, and we develop a theory of normal Cartan
connections, when these conditions are satisfied.  We show how the
treatment from \cite{cap}, about normal Cartan connections of
semisimple type, fits into our theory. We also consider in some detail
the case when $\gh := t^{*}(\gg)$ is  the cotangent Lie algebra of
a non-positively graded Lie algebra $\gg$.\\


\section{Introduction}

The theory of $G$-structures, which is  a coordinate free version
of the {\it r\'ep\'ere mobile} (moving frame) method  by E.
Cartan, provides a powerful tool for the investigation of
different geometric structures on an $n$-dimensional manifold $M$.
 If a geometric  structure, say, a tensor
 field $A$, is infinitesimally homogeneous, i.e. takes the same constant value $A_0$ at any point
$p\in M$ with respect to an appropriate  frame $f_{p} = (e_1,
\cdots , e_n)$, then the set of such adapted frames forms a
$G$-structure $\pi : P \to M$. Applying  the prolongation
procedure to this $G$-structure, one can construct an absolute
parallelism (and, in best cases, a Cartan connection), which can
be used to find invariants for the given geometric structure. More
precisely, recall that the first prolongation of the $G$-structure
$\pi : P \ra M$ is a $G^{(1)}$-structure $\pi^{(1)} : P^{(1)} \ra
P$, with commutative structure group $G^{(1)}= \gg \otimes V^{*}
\cap V \otimes S^{2}(V^{*})$ (where $\gg = \mathrm{Lie}(G)$). The
bundle $\pi^{(1)} : P^{(1)} \ra P$ is not unique: it depends on
the choice of a subspace $D\subset V \otimes S^{2}(V^{*})$,
complementary to $\partial ( \gg \otimes V)$, where $\partial :
\gg \otimes V \ra V \otimes S^{2}(V^{*})$ is the
skew-symmetrization. If the group $G$ is simple, then (in most
cases) $G^{(1)}= \{e\}$  and $\pi^{(1)}: P^{(1)}\to P$ is an
$\{e\}$-structure, or an absolute parallelism, on the bundle of
frames $P$. If moreover $D$ is $G$-invariant, then the $\{
e\}$-structure $\pi^{(1)}: P^{(1)} \ra P$ is identified with a
Cartan connection of type $V \oplus \gg$, i.e. a $G$-equivariant
map $\kappa :TP \to  V\oplus\gg$  such that  $\kappa_p: T_pP \to V
\oplus \gg$ (for any $p\in P$) is an isomorphism, which is  an extension of the
vertical parallelism
 $i_p : T^{\mathrm{vert}}_pP \to \gg$.
 In the more general case  when $\pi: P \to M$ is a $G$-structure of
finite  type, i.e. the $k$-th prolongation
 $G^{(k)}= \gg \otimes S^{k}(V^{*}) \cap V \otimes S^{k+1}(V^{*})$, for  some  $k >0$, is  trivial,
 we get an absolute parallelism on the $k$-th prolongation
 $\pi^{(k)} : P^{(k)} \to P^{(k-1)}$  of the $G$-structure $\pi : P \ra M$ (which,  
in general,  is not  a Cartan connection).

The approach described above does not work if the $G$-structure
has infinite  type, i.e. the full prolongation
$\mathfrak{g}^{(\infty)}$ is infinite dimensional (e.g. for
symplectic or contact structures). To overcome this difficulty, N.
Tanaka developed in \cite{ta1,ta2} a generalization of the theory
of $G$-structures to "non-holonomic $G$-structures" (called Tanaka
structures in \cite{spiro} and infinitesimal flag structures in
\cite{cap}), which are roughly speaking principal subbundles of
frames on a non-holonomic distribution. Tanaka defined the
prolongation of non-positively graded Lie algebras and under some
assumptions, he associated to a non-holonomic $G$-structure a
canonical Cartan connection. There are several expositions of
different versions and generalizations of Tanaka theory
\cite{spiro,cap,sch,morimoto, tanaka-ad,zelenko}. The aim of this
paper is to give a self contained exposition of the Tanaka theory
and to investigate relations between Tanaka structures  and Cartan
connections. We generalize the results by A. ${\check C}$ap and J.
Slovak \cite{cap}, where the special case of parabolic geometry is
studied in detail.

The paper is structured as follows. In Section \ref{prelim},
intended to fix notation,  we recall basic definitions about
graded Lie algebras, Tanaka structures and Cartan connections. Our
approach follows closely \cite{spiro} and \cite{cap}.

In  Section  \ref{cartan-tanaka-inv} we study the relation between
Cartan connections and Tanaka structures. We show that a regular
Cartan connection $\kappa \in \Omega^{1} (P, \gh )$ on a principal
$Q$-bundle $\pi : P \ra M$ induces a Tanaka structure $({\mathcal
D}, \pi_{G} : P_{G} \ra M)$ and conversely, that a Tanaka
structure is always induced by a regular Cartan connection (see
Propositions \ref{cor-tanaka} and \ref{existence-parabolic}). In
the semisimple case (i.e. $\gh $ semisimple with a $|k|$-gradation
and $\mathrm{Lie}(Q) = \gq = \gh_{0} \oplus \gh_{+}$), this was
done in \cite{cap}. We remark that the arguments from \cite{cap}
used to prove these statements require no semisimplicity
assumptions. We simplify and adapt them to our more general
setting.

In  Section  \ref{normal} we  generalize   the  theory  of 
{\bf normal} Cartan connections,  defined  in \cite{cap} for Cartan connections of
semisimple  type, to the non-semisimple case. Let $\pi : P \ra M$ be a principal
$Q$-bundle and  $\kappa : TP \to \gh$  a Cartan connection,
 where $\gh $ is  a graded  semisimple Lie  algebra  with
 non-negative   subalgebra
$\gq = \mathrm{Lie}(Q) = \gh_{0}\oplus \gh_{+}$. Using the
(non-degenerate) Killing form $B$ and an appropriate  Cartan
involution $\theta$ of $\gh$, one  can define a Euclidean metric $B_{\theta} = -
B(\theta \cdot , \cdot )$ adapted to the gradation, see
(\ref{angle}). It turns out that the codifferentials $\partial^{*} :
C^{k+1} (\gh_{-}, \gh )\ra C^{k} (\gh_{-}, \gh )$ (defined as the
metric adjoints of the Lie algebra differentials $\partial : C^{k}
(\gh_{-}, \gh )\ra C^{k+1} (\gh_{-}, \gh )$, with respect to the
metric induced by $B_{\theta}$ on $\{ C^{k}(\gh_{-}, \gh ),\ k\geq
0\}$) are ${Q}$-invariant (with $Q$-action induced by the adjoint
action). This invariancy is the crucial property of  
the theory of normal Cartan connections, which are 
defined   as  Cartan  connections   with   coclosed
 curvature  form, see \cite{cap}. Using  the  semisimple  case
 as motivation, in Subsection \ref{sect-inv} we consider $\gh$ a graded
(not necessarily semisimple) Lie algebra  with a fixed
Euclidean metric $\langle \cdot , \cdot
\rangle$, adapted to the gradation.
 We call the metric
$\langle \cdot , \cdot \rangle$ admissible if the codifferentials
$\partial^{*}: C^{k+1} (\gh_{-}, \gh ) \ra C^{k} (\gh_{-}, \gh )$
($k\geq 0$) are $Q$-invariant (see Definition
\ref{def-admissible}). We characterize admissible metrics (see
Proposition \ref{criterion}). As a consequence, we reobtain in our
setting that the standard metric $B_{\theta}$ on a graded
semisimple Lie algebra is admissible and we show that admissible
metrics exist also on (graded) non-semisimple Lie algebras (see
Corollary \ref{check-ss} and Remark \ref{non-ss}). We develop a
theory of normal Cartan connections of type $\gh$, where $\gh$ is
a graded Lie algebra with an admissible metric, and we show that
various facts from the theory of normal Cartan connections of
semisimple type \cite{cap} are preserved in this more general
setting. In particular, any Tanaka structure which is induced by a
regular Cartan connection of type $\gh$ is induced also by a
normal Cartan connection of type $\gh$ and the cohomology group
$H^{1}_{\geq 1} (\gh_{-}, \gh )$ is the only obstruction for the
uniqueness (up to bundle automorphisms) of a normal Cartan
connection of type $\gh$ inducing a given Tanaka structure (see
Theorems \ref{existence} and \ref{uniqueness}). Our proofs for
these two theorems are based on the arguments from \cite{cap}. We
simplify these arguments and show that no semisimplicity
assumptions are required. While for most graded semisimple Lie
algebras $\gh$, the cohomology group $H^{1}_{\geq 1}(\gh_{-}, \gh
)$ is trivial, our computations from the next section  show
that this is not true in general.

In Section \ref{examples}, which is entirely algebraic, we
consider in detail the cotangent Lie algebra $\gh  = t^{*}(\gg )$
of a non-positively graded Lie algebra $\gg$. We show that $\gh$
inherits a gradation from the gradation of $\gg$, with negative
part $\gg_{-}$. In Subsection \ref{cohomology-ctg}  we compute the
cohomology group $H^{1}_{\geq 1}(\gh_{-}, \gh )$. In general, it
is non-trivial. As an illustration of our computations, in
Subsection \ref{exemplu} we describe $H^{1}_{\geq 1}(\gh_{-}, \gh
)$, in the simplest case when $\gg = \gg_{-1} \oplus \gg_{0}$ has
a non-positive gradation of depth one. Cohomology groups for
cotangent Lie algebras are of independent interest and appear in
the literature in various settings (see e.g. \cite{tirao}, for the
cohomology of the cotangent bundle of Heisenberg group). In
Subsection \ref{metrics} we return to the topic of admissible
metrics.We assume that $\gg$ has a Euclidean metric $\langle \cdot ,
\cdot \rangle_{\gg}$, adapted to the gradation,  and we define a
metric $\langle \cdot , \cdot \rangle$ on $\gh$, which coincides
with the given metric  $\langle \cdot ,
\cdot \rangle_{\gg}$  on $\gg$ and $\gg^{*}$ (identified with
$\gg$ using $\langle\cdot , \cdot \rangle_{\gg}$) and such that
$\gg$ is orthogonal to $\gg^{*}$ with respect to $\langle\cdot ,
\cdot\rangle$. Our motivation to consider such a metric on $\gh$
comes from its formal similarity with the standard metric
$B_{\theta}$ on semisimple Lie algebras. More precisely, both
$\langle \cdot , \cdot \rangle$ and $ B_{\theta}$ are of the form
$B(\theta (\cdot ), \cdot )$, where $\theta : \gh \ra \gh$ is
linear, bijective, and $B\in S^{2}(\gh^{*})$ is a bi-invariant
non-degenerate symmetric form: when $\gh = t^{*}(\gg )= \gg \oplus
\gg^{*}$ is the cotangent Lie algebra, $B$ is the standard metric
of neutral signature of $\gh$ and $\theta : \gg\oplus\gg^{*} \ra
\gg\oplus \gg^{*}$, $\theta (\gg ) = \gg^{*}$, $\theta (\gg^{*} )
= \gg$ is induced by the Riemannian duality defined by
$\langle\cdot , \cdot \rangle_{\gg}$; when $\gh$ is semisimple,
$B$ the Killing form of $\gh$ and $\theta : \gh \ra \gh$ is
(minus) a Cartan involution. We prove that there are strong
obstructions for $\langle \cdot , \cdot \rangle$ to be admissible 
(see Proposition \ref{criterion-adm} and Corollary \ref{consecinta}). The problem to
find admissible metrics on cotangent Lie algebras
remains open and will be studied in the future.

At the end of this introduction, we mention that the general
question of existence of a canonical  Cartan connection associated
to a Tanaka structure was considered, in large generality, also in the
paper \cite{morimoto1} by T. Morimoto,  in the framework  of  his
remarkable  theory of filtered  manifolds. There it was associated,
to any Tanaka structure of type $(\gm, G)$, for which the so called
(C)-condition is satisfied (see Definition 3.10.1 of
\cite{morimoto1}), a canonical Cartan connection $\theta$ (see
Theorem 3.10.1 of \cite{morimoto1}). The principal bundle $P$ on
which $\theta$ is defined was constructed by an elaborated,
induction procedure. The Lie algebra of the structure group $K$ of
$P$ is the non-negative part of the prolongation $(\gm \oplus \gg
)^{\infty}$ of $\gm \oplus \gg$ (when this prolongation is finite
dimensional). The canonical Cartan connection $\theta$ takes values
in $(\gm \oplus \gg )^{\infty}$. Now, as remarked in the proof of
Proposition 3.10.1 of \cite{morimoto1}, if $(\gm \oplus \gg )^{\infty}$
admits a Euclidean metric $\langle \cdot , \cdot \rangle$, adapted
to the gradation, and such that the codifferentials $\partial^{*} :
C^{k+1} (\gm , (\gm \oplus \gg )^{\infty}) \ra  C^{k} (\gm ,(\gm \oplus \gg
)^{\infty})$  ($k\geq 0$) are $K$-invariant (i.e., in our language,
$\langle \cdot , \cdot \rangle$ is admissible), then the condition
(C) is satisfied. Note  that  Morimoto gives  a general
construction of a  canonical  Cartan   connection   associated
 to a Tanaka  structure, but  he  does not   consider  normal Cartan  connections,  i.e.
connections with coclosed curvature. While Morimoto's theory covers a large class of situations, it is
also much more technical compared to our simple, direct approach.\\

{\bf Acknowledgements.} D.V.A. is partially supported by the project
C.Z.1.07/2.3.00/20.0003 of the Operational Programme Education for
Competitiveness of the Ministery of Education, Youth and Sports of
the Czech Republic. L.D. is partially supported by the Romanian
National Authority for Scientific Research, CNCS-UEFISCDI, project
no. PN-II-ID-PCE-2011-3-0362.

\section{Preliminary material}\label{prelim}

In this section, intended to fix notation,  we recall the basic
facts we need about Tanaka structures and Cartan connections (see
e.g. \cite{spiro,cap}).

\subsection{Tanaka structures}

\subsubsection{Graded Lie algebras}

Let $\gh$ be a finite dimensional Lie algebra. A decomposition
$\gh = \gh_{-k}\oplus\cdots \oplus\gh_{\ell}$ (where $k>0$ and
$\ell\geq 0$) of $\gh$ is called a {\bf gradation of depth $k$} if
$[\gh_{i}, \gh_{j}]\subset \gh_{i+j}$,  for any $i,j$ (with the
convention that $\gh_{i}=0$ when $i< -k$ and $i>l$). We will
always assume that such a gradation is {\bf fundamental}, i.e. the
subalgebra $\gh_{-}:= \gh_{-k}\oplus \cdots \oplus \gh_{-1}$ is
generated by $\gh_{-1}$, the subspaces $\gh_{0}$, $\gh_{-k}$ and
$\gh_{\ell}$ are non-trivial and the adjoint action of $\gh_{0}$
on $\gh_{-1}$ is exact. Let $\gh^{i}:= \oplus_{j\geq i} \gh_{j}$
be the associated filtration.

The space $C^{i}(\gh_{-}, \gh )= \Lambda^{i}(\gh_{-}^{*})\otimes
\gh$ of $\gh$-valued $i$-forms on $\gh_{-}$ inherits  a gradation
$C^{i}(\gh_{-}, \gh )=\oplus_{j} C^{i}_{j}(\gh_{-}, \gh )$, where
$$
C^{i}_{j}(\gh_{-}, \gh ):= \sum_{s_{1},\cdots
,s_{i}}(\gh_{s_{1}}\wedge \cdots\wedge \gh_{s_{i}})^{*}\otimes
\gh_{s_{1}+\cdots +s_{i}+j}
$$
is the space of $i$-forms of homogeneous degree $j$. We shall use
the notation
$$
\mathrm{gr}_{j}: C^{i}(\gh_{-}, \gh ) \ra C^{i}_{j}(\gh_{-}, \gh
), \quad \phi \ra \mathrm{gr}_{j} (\phi ) = \phi_{j}
$$
for the natural projection and $C^{i}(\gh_{-}, \gh )^{m}
=\oplus_{j\geq m} C^{i}_{j}(\gh_{-} , \gh )$ for the filtration
associated to the gradation. The Lie algebra differential
$\partial : C^{k}(\gh_{-}, \gh )\ra C^{k+1}(\gh_{-}, \gh )$
defined by
\begin{align}
\label{partial-exp}&(\partial \phi )(X_{0}, \cdots  , X_{k}):=
\sum_{0\leq i\leq k} (-1)^{i} [X_{i}, \phi (X_{0}, \cdots ,
\widehat{X_{i}}, \cdots ,
X_{k}) ] \\
\nonumber & + \sum_{0\leq i<j\leq k} (-1)^{i+j} \phi \left(
[X_{i}, X_{j}], X_{0}, \cdots , \widehat{X_{i}}, \cdots ,
\widehat{X_{j}}, \cdots , X_{k} \right)
\end{align}
(for $X_{i}\in \gh_{-}$, $0\leq i \leq k$) is homogeneous of
degree zero, i.e. it preserves gradations (the hat means that
the argument is omitted).

\subsection{Definition of Tanaka structure}

Let $\mathcal D$ be a distribution on a manifold $M$. It
determines a filtration
\begin{equation}\label{filtration-tan}
\cdots {\mathcal D}_{-d-1} (p) = {\mathcal D}_{-d}(p) \supset
{\mathcal D}_{-d+1}(p) \supset  \cdots \supset {\mathcal
D}_{-1}(p) = {\mathcal D}(p)
\end{equation}
of any tangent space $T_{p}M$, where each subspace ${\mathcal
D}_{-j}(p)$ is spanned by ${\mathcal D} (p)$ and the values at $p$
of  commutators of vector fields $X_{1}, \cdots , X_{s}$ ($2\leq s
\leq j$) in $\mathcal D .$ The commutator of vector fields induces
the structure of a graded Lie algebra on each graded vector space
$$
\gm (p) = \mathrm{gr}_{\mathcal D} (T_{p}M) = \gm_{-d}(p)
\oplus\cdots \oplus \gm_{-1}(p)
$$
associated to the filtred space $T_{p}M$, where $\gm_{-1} (p) =
{\mathcal D}_{-1}(p)$ and $\gm_{-i}(p) = {\mathcal D}_{-i}(p)/
{\mathcal D}_{-i+1} (p)$, for any $2\leq i\leq d.$

A distribution $\mathcal D$ is called {\bf regular  of type $\gm$
(and  depth $k$)} if all Lie algebras $\gm (p)$, $p\in M$, are
isomorphic to a given negatively graded fundamental Lie algebra
$\gm = \gm_{-k}\oplus\cdots \oplus \gm_{-1}.$ If $\mathcal D$ is
regular of type $\gm$, an isomorphism $\xi :{\gm}_{-1}\ra
{\mathcal D} (p)$ is called an {\bf adapted frame} if it can be
extended to an isomorphism $\hat{\xi}: \gm\ra
\mathrm{gr}_{\mathcal D}(T_{p}M)$ of graded Lie algebras. We
remark that if such an extension $\hat{\xi}$ exists, then it is
unique. The group $\mathrm{Aut}_{\mathrm{gr}}(\gm )$ of (grading
preserving) automorphisms of the Lie algebra $\gm$ acts simply
transitively on the set $\mathbb{L}(M, {\mathcal D})$ of all
adapted frames. The natural projection $\pi_{0} : \mathbb{L}(M,
{\mathcal D}) \ra M$ is a principal
$\mathrm{Aut}_{\mathrm{gr}}(\gm )$-bundle, called the {\bf bundle
of $\mathcal D$-frames}.

\bd Let $G$ be a Lie subgroup of $\mathrm{Aut}_{\mathrm{gr}}(\gm
).$ A {\bf Tanaka structure of type $(\gm, G)$} on $M$ is a
regular distribution $\mathcal D$ of type $\gm$ together with a  principal
$G$-subbundle $\pi_{G} : P_{G}\subset \mathbb{L}(M, {\mathcal
D})\rightarrow M$ of the bundle of $\mathcal D$-frames. \ed

Let $\gg : =\mathrm{Lie}(G).$ 
We  denote by $\gm ( \gg )= \gm \oplus \gg$ the
non-positively graded Lie algebra associated to $(\gm, G)$ and by $\gm (\gg )^{\infty}$
its Tanaka prolongation, defined as the maximal graded Lie algebra
with non-positive part isomorphic to $\gm (\gg )$ (see \cite{ta2}).

\subsection{Cartan connections}

\subsubsection{Definition of Cartan connections}

Let $H$ be a Lie group and $Q$ a subgroup of $H$. We denote by
$\gh$ and $\gq$ the Lie algebras of $H$ and $Q$ respectively. Let
$\pi : P \ra M$ be a principal $Q$-bundle. For $a\in \gq$, we
denote by $X^{a}(p) = p \cdot \mathrm{exp} (a)$ the fundamental
vector field on $P$ generated by $a$.

\bd A $1$-form $\kappa\in \Omega^{1}(P, \gh )$ is called a {\bf
Cartan connection of type $\gh$} if:

i) $\kappa (X^{a})=a  $, for  any  $a \in    \gq$;

ii) $\kappa\vert_{T_{u}P} : T_{u} P \ra \gh$ is an isomorphism,
for any $u\in P$;

iii) $(r^{g})^{*}(\kappa ) = \mathrm{Ad}(g^{-1})\circ \kappa$, for
any $g\in Q$.

\ed

The following simple lemma will be useful for our purposes.

\bl\label{tangent-as} The map
\begin{equation}
P\times_{Q} ( \gh / \gq ) \ra TM,\quad [u, \hat{a}] \ra \pi_{*}
(X^{a}_{u}), \quad (u,a)\in P\times \gh
\end{equation}
is an isomorphism. Above  $a\in \gh$ is any representative of
$\hat{a}\in \gh / \gq$ and $X^{a}:= \kappa^{-1}(a)\in {\mathcal
X}(P)$ (is called a {\bf constant vector field}).

\el

The {\bf  curvature} of a  Cartan connection $\kappa$ is a
$2$-form $\Omega\in \Omega^{2}(P, \gh )$ defined by:
$$
\Omega (X,Y)= d\kappa (X,Y) + [\kappa (X), \kappa (Y)] ,\ \forall
X, Y\in TP.
$$
From definitions,
\begin{equation}\label{omega-curv}
\Omega (X^{a}, X^{b}) = - \kappa ( [X^{a}, X^{b}]) + [a, b], \quad
\forall a,b\in \gh ,
\end{equation}
i.e. $\Omega$ measures the failure of $\kappa : {\mathcal X}(P)
\rightarrow \gh$ to be a Lie algebra homomorphism. The curvature
$\Omega\in \Omega^{2}(P, \gh)$ is  invariant, i.e.
($r^{g})^{*}(\Omega ) = \mathrm{Ad}(g^{-1}) \circ \Omega$, for any
$g\in Q$, and horizontal, i.e. $\Omega (X^{a}, \cdot )=0$, for any
$a\in \gq .$ In particular, $\Omega$ is a section of
$\Lambda^{2}(M )\otimes {\mathcal A}(M)$, where ${\mathcal A}(M) =
P\times_{Q}\gh $ is the {\bf adjoint tractor bundle}. The {\bf
curvature function} is defined by
$$
K: P \rightarrow \Lambda^{2}(\gh / \gq )^{*} \otimes\gh , \quad
K(u)( \hat{a}, \hat{b}) = \Omega (X^{a}_{u}, X^{b}_{u}), \quad
u\in P,\ \hat{a}, \hat{b}\in \gh / \gq .
$$

\subsubsection{Cartan connections of graded type}

Let $\kappa\in \Omega^{1}(P, \gh )$ be a Cartan connection on a
principal $Q$-bundle $\pi : P \ra M$. It is called of {\bf graded
type} if a fundamental gradation
\begin{equation}\label{graded-dec}
\gh = \gh_{-k}\oplus \cdots \oplus \gh_{l}= \gh_{-}\oplus \gh_{0}\oplus \gh_{+}
\end{equation}
is given, such that $\gq = \mathrm{Lie}(Q)= \gh_{0} \oplus
\gh_{+}.$ The gradation of $\gh$ induces a gradation of $TP$:
\begin{equation}\label{gradation-P}
TP= (TP)_{-k}\oplus \cdots \oplus (TP)_{\ell},\ (TP)_{i}:=
\kappa^{-1}(\gh_{i}).
\end{equation}
The gradations of $TP$ and $\gh$ allow decompositions of the
curvature $\Omega$ and the curvature function $K$ into homogeneous
components
$$
\Omega = \sum_{i,j,m}\Omega_{ij,m},\quad K = \sum_{i,j,m} K_{ij,m}
$$
where $\Omega_{ij,m}$  (respectively $K_{ij, m}$) are  identified
with sections of $(TP)_{i}^{*}\wedge (TP)_{j}^{*}\otimes \gh_{m}$
(respectively, functions $: P \ra (\gh_{i})^{*} \otimes
(\gh_{j})^{*}\otimes \gh_{m}$) and have homogeneous degree $d:=
m-i-j$.

\bd\label{def-regular} A Cartan connection $\kappa$ is called {\bf
regular} if it is of graded type and all the homogeneous
components of its curvature function $K$ have positive degree.\ed

\section{Cartan connections and Tanaka structures}\label{cartan-tanaka-inv}

\subsection{From  Cartan connections to Tanaka structures}\label{cartan-tanaka}

Let $\gh = \gh_{-k} \oplus \cdots \oplus \gh_{l}$ be a graded Lie
algebra and $\kappa \in \Omega^{1}(P, \gh )$ a Cartan connection
of graded type (not necessarily regular) on a principal $Q$-bundle
$\pi : P\ra M$. Assume that the Lie group $Q$ is decomposed into a
semidirect product $Q = G\cdot N:= H_{0}\cdot H^{+}$, according to
the semidirect decomposition $\gq = \gh_{0}\oplus \gh_{+}$. We
define a $Q$-invariant flag of distributions
\begin{equation}\label{distr-P}
T^{i}P:= \kappa^{-1} (\gh^{i})= \sum_{j\geq i} (TP)_{j}, \quad
-k\leq i\leq l
\end{equation}
of $TP$ associated to the gradation (\ref{gradation-P}). Like in
\cite{cap}, the quotient $P_{G}: = P/ N$, with the natural
projection $\pi_{G}: P_{G}\ra M$ is a principal $G$-bundle. The
natural  projections $\tilde{\pi} : P \ra P_{G}$ and
$\pi_{G}:P_{G}\ra M$ map the flag of distributions on $P$ to a
$G$-invariant flag of distributions
$$
TP_{G}= T^{-k}P_{G}\supset T^{-k+1}P_{G}\supset \cdots\supset
T^{0}P_{G}=T^{\mathrm{vert}}({\pi}_{G})
$$
on $P_{G}$ and to a flag of distributions
\begin{equation}\label{filtration-M}
TM = T^{-k}M\supset T^{-k+1}M\cdots \supset T^{-1}M
\end{equation}
on $M$. We define the graded tangent bundle
$$
\mathrm{gr}(TM)=(TM)_{-k} \oplus (TM)_{-k+1}\oplus\cdots \oplus (TM)_{-1}, \quad
\mathrm{gr}_{-i}(TM)=(TM)_{-i}:= T^{-i}M/ T^{-i+1}M
$$
and we denote by $q_{j}: T^{j}M \ra (TM)_{j}$ the natural
projections. From Lemma \ref{tangent-as},  $T^{i}M \cong
P\times_{Q}( \gh^{i}/ \gq )$, for any $i$. It follows that
\begin{equation}\label{gr-i}
\mathrm{gr}_{i} (TM) \cong P\times_{Q}(\gh^{i}/ \gh^{i+1})=
P_{G}\times_{G}\gh_{i},\ \forall i<0,\quad \mathrm{gr}(TM) \cong
P_{G}\times_{G}\gh_{-}.
\end{equation}
Since $G$ acts on $\gh_{-}$ as a group of automorphisms, the Lie
bracket of $\gh_{-}$ induces a Lie bracket $[ \cdot , \cdot
]_{\gh}$ on $\mathrm{gr}(TM).$ The following proposition can be
proved easily \cite{cap}.

\bp The Cartan connection $\kappa$ is regular if and only if the
following two conditions hold:\

i) The flag of distributions $T^{i}M$ defines a filtration $\Gamma
(T^{i}M)$ of the Lie algebra ${\mathcal X}(M)$ of vector fields on
$M$  (i.e. $[ \Gamma (T^{i}M), \Gamma (T^{j}M) ] \subset \Gamma
(T^{i+j}M)$, for any $i , j$);\

ii)  The bracket $[\cdot , \cdot ] $ on $\mathrm{gr} (TM)$,
induced by the Lie bracket of vector fields, coincides with
$[\cdot , \cdot ]_{\gh}$.\ep

The next proposition is our main result from this section. It
shows that any regular Cartan connection defines a Tanaka structure.

\bp\label{cor-tanaka} Let $\kappa \in \Omega^{1}(P, \gh )$ be a
regular Cartan connection.\

i)  The flag of distributions $TM= T^{-k}M\supset \cdots \supset
T^{-1}M$ defined by $\kappa$ is the derived flag of the
distribution ${\mathcal D}= T^{-1}M$, which is regular of type
$\gh_{-}.$\

ii) The principal $G$-bundle $\pi_{G}: P_{G}\ra M$  is a reduction
to the structure group $G$ of the
$\mathrm{Aut}_{\mathrm{gr}}(\gh_{-})$-bundle of adapted frames
$\pi_{0}: \mathbb{L}(M, {\mathcal D}) \ra M$.\

Equivalently, $( {\mathcal D}, \pi_{G} : P_{G} \ra M)$ is a
Tanaka structure. \ep

\begin{proof} The first claim follows from the assumption that $\gh_{-1}$
generates $\gh_{-}$ and from $[\cdot , \cdot ]_{\gh} = [\cdot ,
\cdot ]$ (because $\kappa$ is regular). 
Since $\mathrm{gr}_{\mathcal D}
(TM) \cong P_{G}\times_{G}\gh_{-}$ (see (\ref{gr-i})), 
the distribution $\mathcal
D$ is regular of type $\gh_{-}$. For the
second claim, we notice that any $p_{0}\in P_{G}$ determines an
isomorphism
\begin{equation}\label{map-need}
\hat{p}_{0}: \gh_{-}\ra \mathrm{gr}_{\mathcal D}
(T_{\pi_{G}(p_{0})}M), \quad a\ra \hat{p}_{0}(a),
\end{equation}
where $\hat{p}_{0}(a) = [p_{0}, a]$ under the identification
$\mathrm{gr}_{\mathcal D}(TM) = P_{G}\times_{G} \gh_{-}$. The map
$$
i: P_{G} \ra \mathbb{L}(M, {\mathcal D}) ,\quad p_{0}\ra
i(p_{0}):= \hat{p}_{0}\in \mathbb{L}_{\pi_{G}(p_{0})}(M,{\mathcal
D})
$$
is injective and covers the natural inclusion
$G\subset \mathrm{Aut}_{\mathrm{gr}} (\gh_{-})$, i.e. is a reduction
of $\pi_{0}: \mathbb{L}(M, {\mathcal D}) \ra M$ to $G$.

\end{proof}

\subsection{The space of Cartan connections inducing a given Tanaka
structure}

We now study the relation between the regular Cartan connections which take
values in the same Lie algebra and induce the same Tanaka
structure. Propositions \ref{diff-curv} and \ref{ajut-adaugat}
below will be our main computational tool in the theory of normal
Cartan connections. Their proofs follow like in the semisimple
case treated in \cite{cap}. To keep the text short, we chose not
to reproduce them here.

Let $\kappa\in\Omega^{1}(P, \gh )$ be a regular Cartan connection
of type $\gh$ on a principal $Q$-bundle $\pi : P \ra M.$ Any other
regular Cartan connection $\tilde{\kappa} \in \Omega^{1} (P, \gh
)$ is given by $\tilde{\kappa } = \kappa +\Phi$, where $\Phi : TP
\ra \gh$ vanishes on $T^{\mathrm{vert}} P = \kappa^{-1} (\gq )$
and will be considered as a function on the quotient $TP /
T^{\mathrm{vert}}P$. Define
\begin{equation}\label{phi}
\phi : P \ra C^{1}(\gh_{-}, \gh ), \quad \phi : = \Phi\circ
\hat{\kappa}^{-1}= (\tilde{\kappa} - \kappa )\circ
\hat{\kappa}^{-1} .
\end{equation}
Above $\hat{\kappa}_{p}^{-1} : \gh_{-} \ra T_{p}P /
T^{\mathrm{vert}}_{p}P$ (for any $p\in P$) is the inverse of the
isomorphism $T_{p}P / T^{\mathrm{vert}}_{p}P \cong \gh_{-}$
induced by $\kappa_{p} : T_{p} P \ra \gh$. The function $\phi$ is
$Q$-invariant. The following holds:

\bp\label{diff-curv} i) The (regular) Cartan connections $\kappa$
and $\tilde{\kappa}$ define the same Tanaka structure if and only
if the function $\phi$ from (\ref{phi}) takes values in
$C^{1}(\gh_{-}, \gh )^{1}$.\

ii) Suppose that $\phi$ takes values in $C^{1}(\gh_{-}, \gh )^{m}$
($m\geq 1$) and let $K$, $\tilde{K}$ the curvature functions of
$\kappa$ and $\tilde{\kappa}.$ Then $\tilde{K} - K$ takes values
in $C^{2}(\gh_{-}, \gh )^{m}$ and
\begin{equation}\label{gr-F}
\tilde{K}_{m} = K_{m} +
\partial (\phi_{m}).
\end{equation}
\ep

For the following proposition, see the proof of Proposition 3.1.14
from \cite{cap} (p. 270).

\bp\label{ajut-adaugat} Let $\tilde{\kappa}, \kappa\in \Omega^{1}
(P, \gh )$ be two regular Cartan connections, which induce the
same Tanaka structure, and $\psi : P \ra \gh^{l}$ a $Q$-invariant
function ($l\geq 1$). Then $\psi^{\mathrm{iso}} (p) := p \cdot
\mathrm{exp} ( \psi (p))$, $p\in P$, is a principal bundle
automorphism and $(\psi^{\mathrm{iso}})^{*} (\tilde{\kappa} )$ is a
regular Cartan connection. Let $\tilde{\phi } : =
((\psi^{\mathrm{iso}})^{*}(\tilde{k}) - \kappa )\circ
\hat{\kappa}^{-1}$ and $\phi : = (\tilde{\kappa} - \kappa )\circ
\hat{\kappa}^{-1} $ the functions associated to the pairs $(
\kappa , (\psi^{\mathrm{iso}})^{*}(\tilde{k}))$ and $(\kappa ,
\tilde{\kappa})$, as in Proposition \ref{diff-curv}. Then
$\tilde{\phi} - \phi +\partial (\psi ): P \ra C^{1} (\gh_{-}, \gh
)^{l+1}.$ In particular, $\tilde{\phi} : P \ra C^{1} (\gh_{-}, \gh
)^{1}$ and $(\psi^{\mathrm{iso}})^{*} (\tilde{\kappa} )$ induces the same
Tanaka structure as $\kappa$ (or $\tilde{\kappa}$).

\ep

\subsection{Construction of a Cartan connection}

In this section we associate to a Tanaka structure $\pi_{G} :
P_{G}\subset \mathbb{L}(M,{\mathcal D})\ra M$ of type $(\gm , G)$
a regular Cartan connection, which induces the given Tanaka
structure. We assume that the Tanaka structure has  finite type,
i.e.  the Tanaka prolongation $\gm (\gg )^{\infty}= \gm \oplus
\gg\oplus \gg_{+}$ of the  non-positively graded Lie algebra $\gm
(\gg )=\gm\oplus \gg$  is finite dimensional. Let $\gn := \gg_{+}=
\gg_{1}\oplus \cdots \oplus \gg_{l}$ be the positive graded part
and $\gq := \gg \oplus \gn$ the non-negative part of $\gm (\gg
)^{\infty}$. We denote by $Q$ the Lie group with Lie algebra $\gq$
and we assume that it is the semidirect product $Q= G\cdot N$ of
closed subgroups $G$ and $N$, generated by $\gg$ and $\gn $
respectively.

We will construct a regular Cartan connection of type $\gm (\gg
)^{\infty }$ on the principal $Q$-bundle
$$
P:= P_{G}\times_{G}Q= P_{G}\times_{G} (G\cdot N ) = P_{G}\times N.
$$
Above $G$ acts on $Q$ by the left action. The (right) $Q$-action
which makes $P= P_{G}\times_{G}Q$ a principal $Q$-bundle is given by\\
$$ q [p_0,q']=   [p_0,q'q],\,\, q
,q' \in Q,\  p_0 \in P_{G}
$$
or in terms of  the representation $P = P_{G} \times N$ by
$$
(g\cdot n)(p_{0}, n^{\prime}) =[p_0, n'gn] =(p_{0}g,
(g^{-1}n^{\prime}g)n), \quad g \in G,\,
 n, n^{\prime}\in N,\ p_{0}\in P_{G}.
$$

\bp\label{existence-parabolic} Let  $({\mathcal D}, \pi_{G} :
P_{G}\ra M)$  be a Tanaka structure. There is a regular
Cartan connection of type  $\gm (\gg )^{\infty }$ on the principal
$Q$-bundle $\pi : P_G \times N \to M$ which induces $({\mathcal D},
\pi_{G})$.\ep

\begin{proof} In a first stage, we construct a Cartan connection of type $\gm (\gg )^{\infty}_{\leq 0}=
\gm \oplus \gg$ on $\pi_{G}:P_{G}\ra M.$ For this, we chose a
principal connection $\omega \in \Omega^{1}(P_{G}, \gg )$, with
horisontal bundle $H\subset TP_{G}$, and a gradation $TM =
(TM)_{-k}\oplus \cdots \oplus (TM)_{-1}$, consistent with the
filtration (\ref{filtration-tan}) determined by $\mathcal D$. (In
particular, $(TM)_{-1} = \mathcal D$). The gradation defines a
bundle isomorphism $i: TM \ra \mathrm{gr}_{\mathcal D}(TM).$ From
the definition of Tanaka structures, $\pi_{G} : P_{G} \ra M$ is a
$G$-subbundle of the principal bundle of $\mathcal D$-frames. In
particular, any point $p_{0}\in P_{G}$ defines an isomorphism
$\hat{p}_{0}:\gm \ra \mathrm{gr}_{\mathcal
D}(T_{\pi_{G}(p_{0})}M).$ By combining the maps $\pi_{G}$, $i$ and
$\hat{p}_{0}$, we obtain an isomorphism
$$
(\kappa_{G})^{\mathrm{hor}}_{p_{0}}: H_{p_{0}}\ra \gm , \quad
(\kappa_{G})^{\mathrm{hor}}_{p_{0}} := (\hat{p}_{0})^{-1}\circ
i\circ (\pi_{G})_{*}.
$$
Extending it to $T_{p_{0}}(P_{G})$ using the vertical parallelism
$T^{\mathrm{vert}}_{p_{0}}(P_{G})\cong \gg$, we obtain an
isomorphism $(\kappa_{G})_{p_{0}} : T_{p_{0}}P_{G}\ra \gm \oplus
\gg .$ It is easy to check that $\kappa_{G}: TP_{G}\ra \gm \oplus
\gg$ constructed in this way is a Cartan connection on $\pi_{G} :
P_{G}\ra M$, of type $\gm (\gg )^{\infty}_{\leq 0} = \gm \oplus
\gg .$

In a second stage, we extend the Cartan connection $\kappa_{G}\in
\Omega^{1}(P_{G}, \gm (\gg )^{\infty}_{\leq 0})$ to a Cartan
connection $\kappa\in \Omega^{1}(P, \gm (\gg )^{\infty} )$ on the
$Q= G\cdot N$-principal bundle $\pi : P= P_{G}\times N \ra M.$
Namely, we define $\kappa = \kappa_{G}\oplus \mu_{N}$, where
$\mu_{N}\in \Omega^{1}(N, \gn )$ is the left-invariant Maurer
Cartan form on $N$. It is easy to check that $\kappa$ is a Cartan
connection which induces the given Tanaka structure.
\end{proof}

\section{Normal Cartan connections}\label{normal}

\subsection{Admissible metrics on graded Lie algebras}\label{sect-inv}

Let $\gh = \gh_{-}\oplus \gh_{0}\oplus \gh_{+}$ be a graded Lie
algebra, $H$ a connected Lie group, with Lie algebra $\gh$, $G=
H_{0}$ and $Q$ connected subgroups with Lie algebras $\gh_{0}$ and
$\gq = \gh_{0}\oplus \gh_{+}$ respectively. We fix a Euclidean metric $\langle \cdot ,
\cdot \rangle$ on $\gh$, {\bf adapted to the gradation}, i.e.
$\langle \gh_{i}, \gh_{j}\rangle =0$ for $i\neq j.$ It induces, in
a natural way, a Euclidean metric, also denoted by $\langle \cdot,
\cdot \rangle$, on the spaces $C^{k}(\gh_{-},\gh )$, $(k\geq 0)$.
We denote by $\partial^{*} : C^{k+1}(\gh_{-},\gh ) \ra
C^{k}(\gh_{-}, \gh )$ the codifferential, or the metric adjoint of
the Lie algebra differential $\partial : C^{k}(\gh_{-}, \gh ) \ra
C^{k+1}(\gh_{-}, \gh )$ defined by (\ref{partial-exp}). Since
$\partial$ preserves the homogeneous degree and $\langle\cdot ,
\cdot \rangle$ is adapted to the gradation,  $\partial^{*}$
preserves homogeneous degree, too.

We  consider $\gh_{-}\cong \gh / \gq$  and its dual $\gh_{-}^{*}$
as $Q$-modules, with $Q$-action induced by the adjoint action of
$Q$ on $\gh .$ We use the same notation $\mathrm{Ad}^{-}_{g}$ for
the action of $g\in Q$ on $\gh_{-}$ and its dual $\gh_{-}^{*}.$ We
denote by $\mathrm{Ad}_{g}^{*} : \gh \ra \gh$ and
$(\mathrm{Ad}_{g}^{-})^{*}: \gh_{-} \ra \gh_{-}$ the metric
adjoints of $\mathrm{Ad}_{g} : \gh \ra \gh$ and
$\mathrm{Ad}^{-}_{g} :\gh_{-} \ra \gh_{-}$ respectively. 
Similarly, for any $X\in \gh$, we denote by $(\mathrm{ad}_{X})^{*}: \gh\ra \gh$ the metric adjoint of
$\mathrm{ad}_{X} : \gh \ra \gh .$  Let
$\gh_{-} \ni X \ra X^{\sharp} = \langle X, \cdot \rangle\in
\gh_{-}^{*}$ and its inverse $\gh_{-}^{*} \ni \gamma \ra
\gamma^{\flat} \in \gh_{-}.$

For any $k\geq 0$, $C^{k}(\gh_{-}, \gh )$ is a $Q$-module, with
$Q$-action given by:
\begin{equation}\label{rho-relation}
(g\cdot\phi ) (X_{1}, \cdots , X_{k}) = \mathrm{Ad}_{g}\left( \phi
(\mathrm{Ad}^{-}_{g^{-1}}(X_{1}), \cdots ,
\mathrm{Ad}^{-}_{g^{-1}} (X_{k}))\right) ,
\end{equation}
for any $g\in Q$, $\phi \in C^{k}(\gh_{-}, \gh)$ and $X_{i}\in
\gh_{-}$.

\bd\label{def-admissible} The metric $\langle\cdot , \cdot
\rangle$ is  called {\bf admissible} if the codifferentials
$\partial^{*} : C^{k+1} (\gh_{-}, \gh ) \ra C^{k} (\gh_{-}, \gh )$
are $Q$-invariant, for any $k\geq 0$. \ed

The following proposition characterizes admissible metrics.

\begin{Prop}\label{criterion}
i) The codifferential $\partial^{*} : C^{k+1} (\gh_{-}, \gh )
\rightarrow C^{k} (\gh_{-}, \gh )$ has the following expression:
for any $\alpha_{i}\in \gh_{-}^{*}$ and $V\in \gh$,
\begin{align}
\nonumber & \partial^{*} ( (\alpha_{0}\wedge \cdots \wedge
\alpha_{k} ) \otimes V)= \sum_{i=0}^{k} (-1)^{i} (\alpha_{0}\wedge
\cdots \wedge \widehat{\alpha_{i}} \wedge \cdots \wedge
\alpha_{k}) \otimes
(\mathrm{ad}_{\alpha_{i}^{\flat}})^{*} (V)\\
\label{codiff}& + \sum_{i<j} (-1)^{i+j} \left( [
\alpha_{i}^{\flat}, \alpha_{j}^{\flat}]^{\sharp} \wedge
\alpha_{0}\wedge \cdots \wedge \widehat{\alpha_{i}} \wedge \cdots
\wedge \widehat{\alpha_{j}} \wedge \cdots \wedge \alpha_{k}
\right)\otimes V.
\end{align}

ii) The metric $\langle \cdot , \cdot \rangle$ is admissible if
and only if
\begin{equation}\label{ad-star}
\mathrm{Ad}_{g}^{*}([Z,W])= [(\mathrm{Ad}^{-}_{g})^{*}(Z),
(\mathrm{Ad}_{g})^{*}(W)], \quad\forall g\in Q, \ \forall  Z\in
\gh_{-},\ \forall W\in \gh
\end{equation}
and
\begin{equation}\label{ad-star-add}
(\mathrm{Ad}^{-}_{g})^{*}([Z,W]) = [(\mathrm{Ad}^{-}_{g})^{*}(Z),
(\mathrm{Ad}^{-}_{g})^{*}(W)], \quad\forall g\in Q, \ \forall
Z,W\in \gh_{-}.
\end{equation}

\end{Prop}

\begin{proof}
Claim {\it i)} follows from direct computation. For claim {\it
ii)}, we notice, from (\ref{rho-relation}) and (\ref{codiff}),
that for any $\alpha : = (\alpha_{0}\wedge \cdots \wedge
\alpha_{k} ) \otimes V\in C^{k+1} (\gh_{-}, \gh )$ and $g\in Q$, 
\begin{align*}
&\partial^{*}(g\cdot\alpha ) =\\
&   \sum_{i} (-1)^{i} \left(
(\mathrm{Ad}^{-}_{g})(\alpha_{0})\wedge \cdots \wedge \widehat{
(\mathrm{Ad}^{-}_{g})(\alpha_{i})}\wedge\cdots \wedge
(\mathrm{Ad}^{-}_{g})(\alpha_{k}) \right) \otimes
(\mathrm{ad}_{(\mathrm{Ad}_{g}(\alpha_{i}))^{\flat}})^{*}
\mathrm{Ad}_{g}
(V)+\\
&  \sum_{i<j} (-1)^{i+j} \left(
[(\mathrm{Ad}_{g}(\alpha_{i}))^{\flat},
(\mathrm{Ad}_{g}(\alpha_{j}))^{\flat}]^{\sharp}\wedge
\mathrm{Ad}_{g}^{*}(\alpha_{0}) \cdots
\widehat{\mathrm{Ad}_{g}(\alpha_{i})} \cdots
\widehat{\mathrm{Ad}_{g}(\alpha_{j})}\cdots\wedge
\mathrm{Ad}_{g}(\alpha_{k})\right) \otimes \mathrm{Ad}_{g} (V).
\end{align*}
Similarly,
\begin{align*}
&g\cdot (\partial^{*}\alpha ) =  \sum_{i} (-1)^{i} \left(
\mathrm{Ad}_{g}(\alpha_{0})\wedge \cdots \wedge
\widehat{\mathrm{Ad}_{g}(\alpha_{i})}\wedge \cdots \wedge
\mathrm{Ad}_{g}(\alpha_{k})\right) \otimes \mathrm{Ad}_{g}
(\mathrm{ad}_{\alpha_{i}^{\flat}})^{*}(V)+ \\
&\sum_{i<j} (-1)^{i+j} (\mathrm{Ad}_{g}([\alpha_{i}^{\flat},
\alpha_{j}^{\flat}]^{\sharp})\wedge \mathrm{Ad}_{g}(\alpha_{0})
\cdots  \widehat{\mathrm{Ad}_{g}(\alpha_{i})} \cdots
\widehat{\mathrm{Ad}_{g}(\alpha_{j})}\cdots\wedge
\mathrm{Ad}_{g}(\alpha_{k}) ) \otimes \mathrm{Ad}_{g}(V).
\end{align*}
In particular, $\partial^{*}(g\cdot \alpha ) = g
\cdot\partial^{*}(\alpha)$ for any $\alpha \in C^{k+1}(\gh_{-}, \gh
)$ (and $k\geq 0$), if and only if for any $\alpha , \beta \in
(\gh_{-})^{*}$ and $g\in Q$,
\begin{equation}\label{prima}
(\mathrm{ad}_{(\mathrm{Ad}_{g}(\alpha ))^{\flat}} )^{*}
\mathrm{Ad}_{g}(V) = \mathrm{Ad}_{g}
(\mathrm{ad}_{\alpha^{\flat}})^{*}(V), \quad
 [(\mathrm{Ad}_{g}(\alpha ))^{\flat}, (\mathrm{Ad}_{g}(\beta
))^{\flat}]^{\sharp} = \mathrm{Ad}_{g}([\alpha^{\flat},
\beta^{\flat}]^{\sharp}).
\end{equation}
On the other hand, it is easy to check that
\begin{equation}\label{adaugat}
(\mathrm{Ad}_{g}^{-})^{*} \left( (\mathrm{Ad}_{g}(\alpha
))^{\flat}\right) = \alpha^{\flat}, \quad \forall g\in Q,\ \forall
\alpha \in \gh_{-}^{*}.
\end{equation}
Using (\ref{adaugat}), it is easy to see that relations
(\ref{prima}) are equivalent to (\ref{ad-star}) and
(\ref{ad-star-add}). For example, to prove (\ref{ad-star}) we
consider the first relation (\ref{prima}) and we take its inner
product with $W\in \gh .$ This gives $(\mathrm{Ad}_{g})^{*} [
(\mathrm{Ad}_{g} (\alpha ))^{\flat} ,W]= [\alpha^{\flat}
,(\mathrm{Ad}_{g})^{*}(W)] .$ Combining this relation with
(\ref{adaugat}), we obtain (\ref{ad-star}). In a similar way we
obtain (\ref{ad-star-add}).

\end{proof}

As an application of Proposition \ref{criterion}, we now give
examples of admissible metrics. First, we consider the situation
treated in \cite{cap}, namely $\gh = \gh_{-l}\oplus \cdots \oplus
\gh_{l}$ a semisimple Lie algebra with an $|l|$-gradation and the
standard metric $B_{\theta}$, defined in the following way. Let
$B$ be the (non-degenerate) Killing form of $\gh$ and $\theta :
\gh \ra \gh$ a Cartan involution, with $\theta (\gh_{i}) =
\gh_{-i}$, for any $i$ (see \cite{cap}, p. 342). Then $B_{\theta}$
is defined by
\begin{equation}\label{angle}
B_{\theta}(X, Y) := - B(X, \theta (Y)), \quad \forall X, Y\in \gh
.
\end{equation}
It is positive definite and adapted
to the gradation.

\bc\label{check-ss} The metric $\langle \cdot , \cdot \rangle =
B_{\theta}$  is admissible.\ec

\begin{proof} A straightforward  computation which uses  $\theta (\gh_{-}) =
\gh_{+}$, $B(\theta (X), \theta (Y) ) = B(X, Y)$ and
$B(\mathrm{Ad}_{g}(X), \mathrm{Ad}_{g}(Y)) = B(X, Y)$, for any $X,
Y\in \gh$ and $g\in Q$, implies that $\mathrm{Ad}_{g}^{*} = \theta
\mathrm{Ad}_{g^{-1}}\theta$ and $(\mathrm{Ad}_{g}^{-})^{*} =
(\theta \mathrm{Ad}_{g^{-1}}\theta )\vert_{\gh_{-}}.$ Since
$\theta$ and $\mathrm{Ad}_{g^{-1}}$ preserve Lie brackets,
relations (\ref{ad-star}) and (\ref{ad-star-add}) hold.

\end{proof}

\begin{rem}\label{non-ss}{\rm Admissible metrics exist also on
non-semisimple (graded) Lie algebras. For example, let $\gh =
\gh_{-1} \oplus \gh_{0}$ be a (non-semisimple) non-positively
graded Lie algebra of depth one. Suppose there are given
$\mathrm{Ad}_{H_{0}}$-invariant Euclidean metrics $\langle \cdot ,
\cdot \rangle_{-1}$ and $\langle \cdot , \cdot \rangle_{0}$ on
$\gh_{-1}$ and $\gh_{0}$ respectively. They induce a metric
$\langle \cdot , \cdot \rangle$ on $\gh$, with respect to which
$\gh_{-1}$ and $\gh_{0}$ are orthogonal, and an easy computation shows
that  $\mathrm{Ad}_{g}^{*} =\mathrm{Ad}_{g^{-1}}: \gh \ra \gh$ and
$(\mathrm{Ad}_{g}^{-})^{*}= \mathrm{Ad}_{g^{-1}} : \gh_{-1} \ra
\gh_{-1}$, for any $g\in H_{0}$. Therefore, conditions
(\ref{ad-star}) and (\ref{ad-star-add}) are satisfied and $\langle
\cdot , \cdot\rangle$ is admissible.
We also mention that Proposition 3.1.10 of \cite{morimoto1} provides other examples of non-semisimple 
Lie algebras which support admissible metrics (see also Remark 3.10.4 from the same
reference).}

\end{rem}

\subsection{Normal Cartan connections}

In this section we define and study normal Cartan connections of
type $\gh$, where $\gh = \gh_{-}\oplus \gh_{0} \oplus\gh_{+}$ is a
graded Lie algebra, with a fixed admissible metric $\langle \cdot
, \cdot\rangle$. We assume that the Lie group $Q$ with Lie algebra
$\gq = \gh_{0} \oplus \gh_{+}$ is the semidirect product $G\cdot
N$ of the Lie groups $G$ and $N$, with Lie algebras $\gh_{0}$ and
$\gn = \gh_{+}$ respectively. Since $\langle \cdot , \cdot
\rangle$ is positive definite, the induced metric on the complex
$\{ C^{k} (\gh_{-}, \gh ), \ k\geq 0\}$ is also positive definite.
The adjoint action of $G$ induces an action on each
space of this complex. With respect to this action, there
is a $G$-invariant Hodge decomposition
\begin{equation}\label{hodge-dec}
C^{k}(\gh_{-}, \gh ) =
\mathrm{Ker}\left(\Delta\vert_{C^{k}(\gh_{-}, \gh ) }
\right)\oplus
\partial^{*}( C^{k+1}(\gh_{-}, \gh ))  \oplus \partial
(C^{k-1}(\gh_{-}, \gh )),
\end{equation}
where $\partial : C^{k-1} (\gh_{-}, \gh ) \ra C^{k} (\gh_{-}, \gh
)$ denotes as usual the Lie algebra differential, $\partial^{*}$
its (metric) adjoint and $\Delta =
\partial\partial^{*} +\partial^{*}\partial$ the Laplacian.
Since $\langle\cdot , \cdot \rangle$ is admissible, the
codifferentials $\partial^{*} : C^{k+1} (\gh_{-}, \gh ) \ra C^{k}
(\gh_{-}, \gh )$ are ${Q}$-invariant.

\bd\label{def-normal} A regular Cartan connection $\kappa\in
\Omega^{1}(P, \gh )$ on a principal $Q$-bundle $\pi : P \ra M$ is
called {\bf normal} if its curvature function $K : P \ra
C^{2}(\gh_{-}, \gh )$ satisfies $\partial^{*} ( K )=0.$\ed

\begin{rem}\label{comentarii-norm}{\rm From $Q$-invariance, the codifferentials induce
bundle maps $\partial^{*} : \Lambda^{k+1} (M)\otimes {\mathcal
A}(M) \rightarrow \Lambda^{k} (M) \otimes {\mathcal A}(M)$ ($k\geq
0$) between forms on $M$ with values in the adjoint bundle
${\mathcal A}(M) =P \times_{Q} \gh .$ The Cartan connection is
normal if its curvature $\Omega\in \Omega^{2} (M, {\mathcal
A}(M))$ satisfies $\partial^{*} (\Omega  )=0.$ }
\end{rem}

The following simple lemma will play an important role in our
treatment of normal Cartan connections.

\bl\label{modificare}  Let $\kappa \in \Omega^{1} (P, \gh)$ be a
regular Cartan connection on a principal $Q$-bundle $\pi : P \ra
M$ and $({\mathcal D}, \pi_{G} : P_{G} \ra M)$ the induced Tanaka
structure.\

i) Let $f: P \ra C^{k} (\gh_{-}, \gh )^{l}$ be $Q$-invariant. Then
its component of degree $l$ is constant on the orbits of $N$ and
it descends to a ($G$-invariant) function $f_{l} : P_{G}\ra
C^{k}_{l} (\gh_{-}, \gh ).$\

ii)  Let $f: P_{G} \ra C^{k}_{l} (\gh_{-}, \gh )$ be a
$G$-invariant function. Then there is a $Q$-invariant function
$f^{Q} : P \ra C^{k} (\gh_{-} , \gh )^{l}$ such that $(f^{Q})_{l}
= f.$\

iii) Let $f: P \ra C^{k} (\gh_{-}, \gh )$ be $Q$-invariant, such
that $\partial^{*} (f): P \ra C^{k-1} (\gh_{-}, \gh )^{l}$, for
$l\geq 0.$ Then there is $\tilde{f}: P \ra C^{k}(\gh_{-}, \gh
)^{l}$, $Q$-invariant, such that $\partial^{*}( f) =
\partial^{*} (\tilde{f}).$
\el

\begin{proof}
Claim {\it i)} follows from the $Q$-invariance $f(u\cdot g) =
g^{-1}\cdot f(u)$ (for any $u\in P$ and $g\in Q=G\cdot N$) and
from the fact that $N$ acts trivially on $C^{k}_{l} (\gh_{-},\gh )
= C^{k} (\gh_{-}, \gh )^{l}/ C^{k} (\gh_{-}, \gh )^{l+1}.$ For
claim {\it ii)}, consider the following general situation. Assume
that $V$ is a $Q$-module (in our case $V=C^{k}(\gh_{-} , \gh )$),
endowed with a $Q$-invariant filtration $\{ V^{i}\}$ (in our case
$V^{i}= C^{k} (\gh_{-} , \gh )^{i}$) and that $N$ acts trivially
on the associated graded vector space $\mathrm{gr}(V) =
\oplus_{i}V^{i}/ V^{i+1}$ (in our case,  $V^{i}/ V^{i+1}=C^{k}_{i}
(\gh_{-}, \gh )$). The  vector bundle $P\times_{Q} V$ inherits a
filtration $(P\times_{Q}V)^{i}:= P\times_{Q} V^{i}$ and
$\mathrm{gr}_{i}(P\times_{Q}V)= (P\times_{Q} V^{i}/ V^{i+1})$ is
isomorphic to $P_{G}\times_{G} V^{i}/ V^{i+1}.$ Claim {\it ii)}
follows by noticing that sections of $\mathrm{gr}_{i} (P\times_{G}
V)$ can be lifted to sections of $P\times_{Q}V^{i}.$ Claim {\it
iii)} follows from the fact that $\partial^{*}$ is $Q$-invariant
and filtration preserving.
\end{proof}

Our main results from this section are the following two theorems,
about existence and uniqueness of normal Cartan connections which
induce a given Tanaka structure.

\bt\label{existence} Let $({\mathcal D}, \pi_{G} : P_{G}\ra M )$
be a Tanaka structure, which is induced by a regular Cartan
connection $\kappa \in \Omega^{1}(P, \gh ).$ We fix an admissible
metric on $\gh .$ Then the given Tanaka structure is induced also
by a normal Cartan connection $\kappa^{n} \in \Omega^{1}(P, \gh
)$.\et

\begin{proof}
Let $K^{(0)} =\sum_{i\geq 1} K_{i}^{(0)}$ be the curvature
function  of the regular Cartan connection $\kappa^{(0)}:=
\kappa$, where $K_{i}^{(0)}: P \ra C^{2}_{i} (\gh_{-}, \gh )$,
$i\geq 1.$ From Lemma \ref{modificare} {\it i)}, $K_{1}^{(0)}:
P_{G} \ra C^{2}_{1} (\gh_{-}, \gh )$ is $G$-invariant. Using
(\ref{hodge-dec}) for $C^{2}_{1}(\gh_{-}, \gh )$, we can decompose
$G$-invariantly $K_{1}^{(0)} = (K_{1}^{(0)})^{\mathrm{harm}} +
\partial^{*} (\eta_{1}^{(0)} ) +
\partial (\phi_{1}^{(0)} )$,
where $\eta_{1} ^{(0)}: P_{G} \ra  C^{3}_{1}(\gh_{-}, \gh )$ and
$\phi_{1}^{(0)}: P_{G} \ra  C^{1}_{1}(\gh_{-}, \gh )$. From Lemma
\ref{modificare} {\it ii)}, there is $(\phi^{(0)}_{1})^{Q} : P \ra
C^{1} (\gh_{-}, \gh )^{1}$ $Q$-invariant, with the homogeneous
degree one component equal to $\phi^{(0)}_{1}.$ Define a Cartan
connection $\kappa^{(1)}:=\kappa^{(0)} - (\phi^{(0)}_{1})^{Q}\circ
\kappa^{(0)}$ and let $K^{(1)}$ its curvature function. We apply
Proposition \ref{diff-curv} to the pair $(\kappa^{(0)},
\kappa^{(1)})$. Since $(\phi^{(0)}_{1})^{Q}$ takes values in
$C^{1}(\gh_{-}, \gh )^{1}$, we deduce that $\kappa^{(1)}$ induces
the given Tanaka structure, the difference $K^{(1)} - K^{(0)}$
takes values in $C^{2} (\gh_{-}, \gh )^{1}$ and
$$
K^{(1)}_{1}= K^{(0)}_{1} - \partial (\phi_{1}^{(0)} ) =
(K_{1}^{(0)})^{\mathrm{harm}}+\partial^{*}(\eta_{1}^{(0)})
$$
is coclosed. To summarize: $\kappa^{(1)}$ is a regular Cartan
connection, which induces the given Tanaka structure, and whose
curvature function $K^{(1)}$ has the property that its homogeneous
degree one component $K^{(1)}_{1}$ is coclosed (or $\partial^{*}
(K^{(1)}) : P \ra C^{1}(\gh_{-}, \gh )^{2}$).

Using Lemma \ref{modificare} {\it iii)}, let $\widetilde{K^{(1)}}
: P \ra C^{2}(\gh_{-}, \gh )^{2}$ $Q$-invariant, such that
$\partial^{*} (\widetilde{K^{(1)}})=
\partial^{*} (K^{(1)})$. As before, $(\widetilde{K^{(1)}})_{2}: P_{G}\ra C^{2}_{2}
(\gh_{-}, \gh )$ is $G$-invariant and from (\ref{hodge-dec}) for
$C^{2}_{2} (\gh_{-}, \gh )$, we can decompose $G$-invariantly
$(\widetilde{K^{(1)}})_{2}=
(\widetilde{K^{(1)}})_{2}^{\mathrm{harm}} +
\partial^{*}(\eta_{2}^{(1)}) +\partial (\phi_{2}^{(1)}),$ where $\eta_{2}^{(1)} : P_{G} \ra  C^{3}_{2}(\gh_{-}, \gh
)$ and $\phi_{2}^{(1)}: P_{G} \ra  C^{1}_{2}(\gh_{-}, \gh )$. Let
$(\phi_{2}^{(1)})^{Q} : P \ra C^{1} (\gh_{-}, \gh )^{2}$
$Q$-invariant, with the homogeneous degree two component equal to
$\phi_{2}^{(1)}.$ Define a new Cartan connection $\kappa^{(2)} :=
\kappa^{(1)} - (\phi^{(1)}_{2})^{Q}\circ \kappa^{(1)}$ and let
$K^{(2)}$ its curvature function. Again from Proposition
\ref{diff-curv}, applied to $(\kappa^{(1)}, \kappa^{(2)})$,
$K^{(2)} - K^{(1)}$ takes values in $C^{2}(\gh_{-}, \gh )^{2}$ and
$K^{(2)}_{2}=K^{(1)}_{2} - \partial (\phi^{(1)}_{2})$. Thus
$K^{(2)}_{1} = K_{1}^{(1)}$ is coclosed (because $K_{1}^{(1)}$ is
coclosed) and
$$
\partial^{*}( K^{(2)}_{2}) =\partial^{*}( K^{(1)}_{2})  - \partial^{*}\partial (\phi^{(1)}_{2})=
\partial^{*} (\widetilde{K^{(1)}})_{2} - \partial^{*}\partial
(\phi_{2}^{(1)} )= \partial^{*}\partial (\phi_{2}^{(1)} ) -
\partial^{*}\partial (\phi_{2}^{(1)} )=0,
$$
where in the second equality we used $\partial^{*} (K^{(1)})=
\partial^{*}( \widetilde{K^{(1)}})$ and in the third equality
we used the Hodge decomposition of  $(\widetilde{K^{(1)}})_{2}$.
To summarize: $\kappa^{(2)}$ is a regular Cartan  connection,
which induces the given Tanaka structure, and the homogeneous
degree one and two components of its curvature function are
coclosed. Repeating inductively the argument we obtain a normal
Cartan connection $\kappa^{n}\in \Omega^{1} (P, \gh )$, as
required.

\end{proof}

\bt\label{uniqueness} Let $\gh$ be a graded Lie algebra, with
$H^{1}_{\geq 1} (\gh_{-}, \gh )=0$. Fix an admissible metric on
$\gh .$ For any two normal Cartan connections $\kappa,
\tilde{\kappa} \in \Omega^{1}(P, \gh )$, which induce the same
Tanaka structure, there is a bundle automorphism $\psi : P \ra P$
such that $\tilde{\kappa}=\psi^{*} (\kappa ).$

\et

\begin{proof}
Let $K$, $\tilde{K}$, the curvature functions of $\kappa$,
$\tilde{\kappa}$ respectively, and $\phi^{(0)} := (\tilde{\kappa}
- \kappa ) \circ \hat{\kappa}^{-1} : P \ra C^{1} (\gh_{-}, \gh
)^{1}$, which is $Q$-invariant. Then $\phi_{1}^{(0)} : P_{G} \ra
C^{1}_{1} (\gh_{-}, \gh )$ is $G$-invariant and $\tilde{K}_{1} -
K_{1} = \partial (\phi_{1}^{(0)})$ (see Proposition
\ref{diff-curv}). Since $\kappa$ and $\tilde{\kappa}$ are normal,
$\tilde{K}_{1}$, $K_{1}$ and therefore also $\partial
(\phi_{1}^{(0)})$ are coclosed. The Hodge decomposition of
$C^{2}_{1} (\gh_{-}, \gh )$ implies $\partial (\phi_{1}^{(0)})=0.$
Since $H^{1}_{1} (\gh_{-}, \gh ) =0$, there is $\mu_{1}^{(0)} :
P_{G} \ra \gh_{1}$, $G$-invariant, such that $\phi_{1}^{(0)} =
\partial (\mu_{1}^{(0)}).$
Using Lemma \ref{modificare} {\it ii)}, let $(\mu_{1}^{(0)})^{Q} :
P \ra \gh^{1}$, $Q$-invariant, with the homogeneous degree one
component equal to $\mu_{1}^{(0)}$, and define $\psi_{0} : P \ra
P$ by $\psi_{0}(p) := p \cdot \mathrm{exp} (( \mu_{1}^{(0)})^{Q}
(p))$. Then $\psi_{0}^{*} (\tilde{\kappa })$ is a normal Cartan
connection, which induces the given Tanaka structure. We claim
that $\phi^{(1)}:= (\psi_{0}^{*}(\tilde{\kappa}) - \kappa ) \circ
\hat{\kappa}^{-1}$ takes values in $C^{1} (\gh_{-}, \gh )^{2}.$ To
prove this claim, we notice, from  Proposition \ref{ajut-adaugat},
that $\phi^{(1)} - \phi^{(0)} +
\partial ( (\mu_{1}^{(0)})^{Q})$ takes values in $C^{1}(\gh_{-},
\gh )^{2}$. In particular, $\phi^{(1)}_{1} = \phi_{1}^{(0)} -
\partial (\mu_{1}^{(0)}) = 0$. It follows that $\phi^{(1)} : P \ra C^{1}
(\gh_{-}, \gh )^{2}$, as claimed.

To summarize: $\psi_{0}^{*}(\tilde{\kappa})$ is a normal Cartan
connection, which induces the given Tanaka structure, and  the
function $\phi^{(1)}= (\psi_{0}^{*}(\tilde{\kappa}) - \kappa )
\circ \hat{\kappa}^{-1}$ associated to the pair $(\kappa ,
\psi^{*}_{0} (\tilde{\kappa}))$ takes values in $C^{1}(\gh_{-},
\gh )^{2}.$ An argument as before, with $\tilde{\kappa}$ replaced
by $\psi_{0}^{*}(\tilde{\kappa })$ (and $\kappa$ the same), which
uses $H^{1}_{2} (\gh_{-}, \gh )=0$, shows that the homogeneous
degree two component $\phi^{(1)}_{2}$ of $\phi^{(1)}$ is of the
form $\phi^{(1)}_{2} =
\partial (\mu^{(1)}_{2})$, for a $G$-invariant function
$\mu^{(1)}_{2} : P_{G} \ra \gh_{2}.$ Let $(\mu_{2}^{(1)})^{Q} : P
\ra \gh^{2}$, $Q$-invariant, with the homogeneous degree two
component equal to $\mu^{(1)}_{2}$, and define $\psi_{1} : P \ra
P$, by $\psi_{1}(p):= p \cdot \mathrm{exp} (
(\mu_{2}^{(1)})^{Q}(p))$. As before, one shows that the function
$\phi^{(2)}:= (\psi_{1}^{*}\psi_{0}^{*}(\tilde{\kappa}) - \kappa )
\circ \hat{\kappa}^{-1}$ associated to the pair $(\kappa ,
\psi_{1}^{*}\psi_{0}^{*} (\tilde{\kappa}))$ takes values in
$C^{1}(\gh_{-}, \gh )^{3}.$ An induction argument concludes the proof.

\end{proof}

\section{The cotangent Lie algebra}\label{examples}

Let $\gg =\gg_{-}\oplus \gg_{0}=  \gg_{-k}\oplus \cdots\oplus
\gg_{-1} \oplus\gg_{0} $ be a non-positively graded fundamental
Lie algebra.  We shall consider $\gg^{*}$ with the induced
non-negative gradation $(\gg^{*})_{i}: = (\gg_{-i})^{*}$. 
For $X\in \gg$, we denote by $L_{X}$ its coadjoint action on $\gg^{*}$,
given by $L_{X}(\alpha ) (Y) := - \alpha ( [ X, Y])$, for any $\alpha \in \gg^{*}$
and $Y\in \gg .$  The
cotangent Lie algebra $ \gh = t^{*}(\gg )$, defined as the direct
sum $\gg \oplus \gg^{*}$ with the Lie bracket
$$
[X +\xi , Y+\eta ] = [X, Y] + L_{X}(\eta ) - L_{Y}(\xi ),
$$
inherits a natural (fundamental) gradation $\gh
=\gh_{-k}\oplus\cdots \oplus \gh_{0}\oplus \cdots \oplus \gh_{k}$,
where
$$
\gh_{i} = \begin{cases}
\gg_{i}, \quad i <0\\
\gg_{0}\oplus (\gg_{0})^{*}, \quad i=0\\
(\gg_{-i})^{*},\quad i>0.
\end{cases}
$$
In the following section we compute the cohomology group $H^{1}_{\geq 1}
(\gh_{-}, \gh ) = \oplus_{l\geq 1} H^{1}_{l} (\gh_{-}, \gh ).$

\subsection{ The cohomology group $H^{1}_{\geq 1}(\gg_{-}, t^{*}(\gg )
)$}\label{cohomology-ctg}

We begin by fixing notation. We consider $\gh$, $\gg$ and
$\gg^{*}$ as $\gh_{-}$-modules, via the adjoint
representation in the cotangent Lie algebra $\gh$, and we denote
by $\partial_{V}$ the Lie algebra differential in the complex $\{
C^{k}(\gh_{-}, V ), \ k\geq 0\}$, where $V = \gh , \gg$ or
$\gg^{*}.$ We denote by $C^{k}_{l} (\gh_{-}, V)\subset
C^{k}(\gh_{-}, V)$ the subspace of forms of homogeneous degree
$l$. For  $\alpha \in C^{1}(\gh_{-}, \gh )$, we denote by
$\alpha_{\gg}\in C^{1}( \gh_{-} , \gg )$ and $\alpha_{\gg^{*}} \in
C^{1}( \gh_{-}, \gg^{*})$ the composition of $\alpha$ with the
natural projections from $\gh = \gg \oplus \gg^{*}$ onto its
factors $\gg$ and $\gg^{*}.$ The isomorphism
$$
C^{1}(\gh_{-}, \gh )
\ni \alpha \ra (\alpha_{\gg}, \alpha_{\gg^{*}})\in C^{1} (\gh_{-},
\gg ) \oplus C^{1} (\gh_{-}, \gg^{*})
$$
is compatible with homogeneous degree and
\begin{equation}\label{decomp-deriv}
\partial_{\gh} (\alpha ) =
\partial_{\gg} (\alpha_{\gg}) +\partial_{\gg^{*}}
(\alpha_{\gg^{*}}), \quad \forall \alpha\in C^{1} (\gh_{-}, \gh ).
\end{equation}

\bl For any $l\geq 1$,
\begin{equation}\label{decomp-cohom}
H^{1}_{l}( \gh_{-}, \gh ) =\mathrm{Ker} \left(
\partial_{\gg} : C^{1}_{l} (\gg_{-}, \gg )\ra C^{2}_{l} (\gg_{-}, \gg ) \right) \oplus
H^{1}_{l}(\gg_{-}, \gg^{*}).
\end{equation}
\el

\begin{proof} For $\beta \in C^{0}_{l} (\gh_{-}, \gh)= (\gg_{-l})^{*}$, $\partial_{\gh}(\beta
)\in C^{1}_{l} (\gh_{-}, \gh )$ is given by $\partial_{\gh} (\beta
)(X) = L_{X}(\beta )$, for any $X\in \gh_{-}$. In particular,
$\partial_{\gh} (\beta )$ takes values in $\gg^{*}$ and belongs to
$C^{1}_{l} (\gg_{-}, \gg^{*}).$ Our claim follows from
(\ref{decomp-deriv}).
\end{proof}

Therefore, in order to compute   $H^{1}_{l}( \gh_{-}, \gh )$, we
need to determine both terms in the right hand side of
(\ref{decomp-cohom}). This is done in the next two propositions.

\bp\label{calc-g}  i) For any $l\geq 2$,  $\mathrm{Ker}\left(
\partial_{\gg} : C^{1}_{l}(\gg_{-}, \gg ) \ra
C^{2}_{l}(\gg_{-}, \gg )\right) = \{ 0\}$;\

ii) The restriction map $C^{1}_{1} (\gg_{-}, \gg ) \ra
\mathrm{Hom}(\gg_{-1}, \gg_{0})$ defines an injection
$$
\mathrm{Ker}\left( \partial_{\gg} : C^{1}_{1}(\gg_{-}, \gg ) \ra
C^{2}_{1}(\gg_{-}, \gg )\right) \ra \mathrm{Hom}(\gg_{-1},
\gg_{0})
$$
whose image $\mathcal S (\gg) \subset \mathrm{Hom} (\gg_{-1},
\gg_{0})$ is the vector space of all linear functions $ f:
\gg_{-1}\ra \gg_{0}$ with the property: for any $X_{1}, \cdots ,
X_{s},X^{\prime}_{1}, \cdots , X^{\prime}_{s} \in \gg_{-1}$, such
that
$$
\mathrm{ad}_{X_{1}}\cdots \mathrm{ad}_{X_{s-1}} (X_{s}) =
\mathrm{ad}_{X_{1}^{\prime}}\cdots \mathrm{ad}_{X^{\prime}_{s-1}}
(X^{\prime}_{s}),
$$
the following equality holds:
\begin{align}
\nonumber&\mathrm{ad}_{f(X_{1})} \mathrm{ad}_{X_{2}} \cdots
\mathrm{ad}_{X_{s-1}} (X_{s}) + \cdots + \mathrm{ad}_{X_{1}}
\cdots \mathrm{ad}_{f(X_{s-1})} (X_{s}) + \mathrm{ad}_{X_{1}}
\cdots \mathrm{ad}_{X_{s-1}} f(X_{s})\\
\label{ad-cond}& = \mathrm{ad}_{f(X^{\prime}_{1})}
\mathrm{ad}_{X^{\prime}_{2}} \cdots \mathrm{ad}_{X^{\prime}_{s-1}}
(X_{s}^{\prime}) + \cdots + \mathrm{ad}_{X^{\prime}_{1}} \cdots
\mathrm{ad}_{f(X^{\prime}_{s-1})} (X^{\prime}_{s}) +
\mathrm{ad}_{X^{\prime}_{1}} \cdots \mathrm{ad}_{X^{\prime}_{s-1}}
f(X^{\prime}_{s}).
\end{align}
\ep

\begin{proof}
Let $\beta \in C^{1}_{l}(\gg_{-}, \gg )$ with $\partial_{\gg}
(\beta )=0$. Then
\begin{equation}\label{partial-0}
\beta ( [X, Y])= [ X, \beta (Y) ] - [ Y, \beta (X)], \quad \forall
X, Y\in \gg_{-}.
\end{equation}
Suppose first that $l\geq 2$. Since $\beta$ is of homogeneous
degree $l$, $\beta\vert_{\gg_{i}}=0$, for any $i\geq -l+1$. In
particular, $\beta\vert_{\gg_{-1}}=0$. Since $\gg_{-1}$ generates
$\gg_{-}$, relation (\ref{partial-0}) implies $\beta =0$, as
required.

It remains to consider the case  $l=1$. Then $\beta
(\gg_{i})\subset \gg_{i+1}$, for any $i\leq -1$, and from
(\ref{partial-0}) $\beta$ is determined by its restriction
$\beta\vert_{\gg_{-1}} : \gg_{-1} \ra \gg_{0}.$ It is
straightforward to check that a linear map $ f : \gg_{-1} \ra
\gg_{0}$ can be extended to a map $\beta : \gg_{-} \ra \gg$ of
homogeneous degree one, with $\partial_{\gg } (\beta )=0$, if and
only if it satisfies (\ref{ad-cond}). Our second claim follows.

\end{proof}

Next, we need to compute $H^{1}_{l}(\gh_{-}, \gg^{*})$ for $l\geq
1.$ Any $\gamma \in C^{1}_{l}(\gg_{-}, \gg^{*})$ such that
$\partial_{\gg^{*}} (\gamma )=0$, i.e.
\begin{equation}\label{F}
\gamma ([X, Y]) = \gamma (X) \circ \mathrm{ad}_{Y} -\gamma
(Y)\circ \mathrm{ad}_{X}, \quad \forall X, Y\in \gg_{-},
\end{equation}
is determined by its restriction $\gamma\vert_{\gg_{-1}} :
\gg_{-1} \ra (\gg_{1-l})^{*}.$ With this preliminary remark, we
state:

\begin{Prop}\label{calc-g-star}
Let $l\geq 1.$ The restriction map $C^{1}_{l} (\gg_{-},
\gg^{*})\ra \mathrm{Hom} (\gg_{-1}, (\gg_{1-l})^{*})$ induces an
isomorphism
$$
H^{1}_{l} (\gg_{-}, \gg^{*}) = Z_{l}(\gg ) / B_{l} (\gg).
$$
Above $Z_{1}(\gg ) = \mathrm{Hom} (\gg_{-1}, (\gg_{0})^{*})$ and
for $l\geq 2$, $Z_{l}(\gg )\subset \mathrm{Hom} (\gg_{-1},
(\gg_{1-l})^{*})$  is the vector space of all maps $\alpha :
\gg_{-1} \ra (\gg_{1-l})^{*}$ which satisfy: for any $X_{1},
\cdots , X_{s}, X_{1}^{\prime}, \cdots , X_{s}^{\prime} \in
\gg_{-1}$ ($2\leq s\leq l$) with
$$
\mathrm{ad}_{X_{1}}\mathrm{ad}_{X_{2}}\cdots \mathrm{ad}_{X_{s-1}}
(X_{s}) =
\mathrm{ad}_{X^{\prime}_{1}}\mathrm{ad}_{X^{\prime}_{2}}\cdots
\mathrm{ad}_{X^{\prime}_{s-1}} (X^{\prime}_{s}),
$$
the following relation holds:
\begin{align}
\nonumber& \sum_{i=1}^{s} (-1)^{i+1} \alpha (X_{i})
\mathrm{ad}_{\mathrm{ad}_{X_{i+1}}\mathrm{ad}_{X_{i+2}}\cdots
\mathrm{ad}_{X_{s-1}}(X_{s}) } \mathrm{ad}_{X_{i-1}}\cdots
\mathrm{ad}_{X_{1}}\vert_{\gg_{s-l}}\\
\label{ad-cot}& =\sum_{i=1}^{s} (-1)^{i+1} \alpha (X^{\prime}_{i})
\mathrm{ad}_{\mathrm{ad}_{X^{\prime}_{i+1}}\mathrm{ad}_{X^{\prime}_{i+2}}\cdots
\mathrm{ad}_{X^{\prime}_{s-1}}(X^{\prime}_{s}) }
\mathrm{ad}_{X^{\prime}_{i-1}}\cdots
\mathrm{ad}_{X^{\prime}_{1}}\vert_{\gg_{s-l}}.
\end{align}
Moreover, for any $l\geq 1$,  $B_{l}(\gg ) = \{ \hat{\beta} :
\gg_{-1} \ra (\gg_{1-l})^{*}, \ \beta \in (\gg_{-l})^{*}\} $,
where
\begin{equation}\label{beta-hat}
\hat{\beta} (X)(Z) := -\beta ( [X, Z]), \quad \forall X\in
\gg_{-1},\ Z\in \gg_{1-l}.
\end{equation}
\end{Prop}

\begin{proof}
Let $l=1.$ Any $\gamma \in C^{1}_{1} (\gg_{-}, \gg^{*})$ is
trivial on $\gg_{-k}\oplus \cdots \oplus \gg_{-2}\subset \gg_{-}$
and is determined by its restriction $\gamma\vert_{\gg_{-1}} :
\gg_{-1} \ra ( \gg_{0})^{*}.$ Conversely, the trivial extension of
any map $\alpha : \gg_{-1} \ra ( \gg_{0})^{*}$ to $\gg_{-}$
belongs to $C^{1}_{1} (\gg_{-}, \gg^{*})$ and is
$\partial_{\gg^{*}}$-closed.

Let $l\geq 2$ and $\gamma \in C^{1}_{l} (\gh_{-},\gg^{*})$ such
that $\partial_{\gg^{*}} (\gamma )=0.$ Relation (\ref{F}) and an
induction argument shows that the restriction $\alpha  : = \gamma
\vert_{\gg_{-1}} : \gg_{-1} \ra (\gg_{1-l})^{*}$ satisfies
(\ref{ad-cot}). Conversely, it may be shown that any linear map
$\alpha  :  \gg_{-1} \ra (\gg_{1-l})^{*} $ which satisfies
(\ref{ad-cot}) can be extended (uniquely) to an homogenous degree
$l$  linear map $\gamma :\gg_{-} \ra \gg^{*}$, which satisfies
$\partial_{\gg^{*}} (\gamma )=0.$

So far, we proved that for any $l\geq 1$, the restriction
$C^{1}_{l} (\gg_{-}, \gg^{*}) \ra \mathrm{Hom} (\gg_{-1},
(\gg_{1-l})^{*})$  maps $\mathrm{Ker}(\partial_{\gg^{*}})$
isomorphically onto $Z_{l} (\gg ).$ To conclude the proof, we need
to show that it also maps $\mathrm{Im} (\partial_{\gg^{*}})$ onto
$B_{l} (\gg )$. For this, we first notice that for any $\beta \in
C^{0}_{l} (\gg_{-}, \gg^{*}) = (\gg_{-l})^{*}$, $\hat{\beta}$
defined in (\ref{beta-hat}) is equal to $\partial_{\gg^{*}}(\beta
) \vert_{\gg_{-1}}.$ Consider now $\gamma \in C^{1}_{l} (\gg_{-},
\gg^{*})$ with $\partial_{\gg^{*}} (\gamma )=0$. We need to show
that $\gamma\vert_{\gg_{-1}} = \hat{\beta}$ implies $\gamma =
\partial_{\gg^{*}} (\beta )$. This follows by an induction argument. More precisely, assume that
$\gamma\vert_{\gg_{-1}}  = \hat{\beta}.$ Then, for any $X, Y\in
\gg_{-1}$ and $Z\in \gg_{2-l}$,
\begin{align*}
\gamma ([ X, Y] ) (Z) &=\gamma (X) ( [ Y, Z]) - \gamma (Y) ( [X,
Z]) = - \beta ( [ X, [ Y, Z]]) + \beta ( [ Y, [X, Z]])\\
& = - \beta ( [X, Y], Z]) = \partial_{\gg^{*}} (\beta ) ( [X, Y])
(Z),
\end{align*}
i.e.  $\gamma\vert_{\gg_{-2} } = \partial_{\gg^{*}} (\beta
)\vert_{\gg_{-2}}$. By induction, we obtain $\gamma
=\partial_{\gg^{*}} (\beta )$, as required.
\end{proof}

As a consequence of our above considerations we obtain:

\bt\label{adunat} Let $l\geq 1$ and $B_{l}(\gg )$, $Z_{l}(\gg )$
and $S (\gg )$, as in Propositions \ref{calc-g} and
\ref{calc-g-star}. Then
$$
H^{1}_{1}(\gh_{-}, \gh )  = {\mathcal S}(\gg ) \oplus Z_{1}(\gg )
/ B_{1}(\gg ),\ H^{1}_{l} (\gh_{-}, \gh ) = Z_{l}(\gg )/ B_{l}(\gg
), \quad l\geq 2.
$$
\et

\subsection{Example: $|1|$-graded cotangent Lie algebra
$t^{*}(\gg)$}\label{exemplu}

As an illustration of the above computations, we consider now the
simple case when $\gg = \gg_{-1}\oplus\gg_{0} = V \oplus
\mathfrak{g}_0 $ has a non-positive gradation of depth one defined
by  a representation  $\mu : \mathfrak{g}_0 \otimes V \to V$ of a
Lie  algebra $\mathfrak{g}_0$ on a vector  space $V =
\mathfrak{g}_{-1}$. Then $\mathfrak{h}= t^{*}(\gg )$ has a
$|1|$-gradation
$$
t^{*}(\gg ) =\gh = \gh_{-1} \oplus \gh_{0} \oplus \gh_{1}=
\gg_{-1} \oplus ( \gg_{0} \oplus (\gg_{0})^{*}) \oplus
\gg_{-1}^{*} = V \oplus ( \gg_{0} \oplus (\gg_{0})^{*})\oplus
V^*.
$$
We denote  by $\mu^* : V^* \to (\mathfrak{g}_0\otimes V)^*$ the
dual map of $\mu$ and   by $ \mathfrak{g}_0^{[1]}:=
\mathfrak{g}_0\otimes V^* \cap V \otimes \Lambda^2V^*$ the first
skew-prolongation of $\mathfrak{g}_0$, considered as  a linear Lie
algebra. (For descriptions of skew-prolongations and various
applications, see e.g. \cite{nagy}). We denote by
$(\Lambda^2V^*)^{\mathfrak{g}_0}$ the space of
$\mathfrak{g}_0$-invariant skew-symmetric forms and by
$(S^2V^*)_{\mathfrak{g}_0}$ the  space of  symmetric bilinear
forms with respect to which all endomorphisms from
$\mathfrak{g}_0$ are  symmetric. From our previous computations we
obtain:

\begin{Prop} The cohomology groups $H^{1}_{l} (V, \mathfrak{h})$ are
trivial for $l\geq 3$ and
$$
H_{1}^{1} (V, \mathfrak{h})=  \mathfrak{g}_0^{[1]}\oplus (V
\otimes \mathfrak{g}_0)^*/\mu^*(V^*), \quad H^{1}_{2}(V, h ) =
(\Lambda^2V^*)^{\mathfrak{g}_0}\oplus( S^2V^*)_{\mathfrak{g}_0}.
$$

\end{Prop}

For most linear Lie  algebras $\mathfrak{g}_0$,
$(\Lambda^2V^*)^{\mathfrak{g}_0}, \, ( S^2V^*)_{\mathfrak{g}_0}$
and $\gg_{0}^{[1]}$ are trivial. However, $\mu^*(V^*)\neq   V
\otimes (\mathfrak{g}_0)^*$, if $\dim \mathfrak{g}_0 >1 $. Thus,
if $\dim \mathfrak{g}_0 >1 $, then $H_1^{1} (V, \mathfrak{h})\neq
0$.

\subsection{Remarks on admissible metrics on $t^{*}(\gg )$}\label{metrics}

Let $\langle\cdot , \cdot \rangle_{\gg}$  be a
Euclidean, adapted to the gradation, metric on  $\gg = \gg_{-k}\oplus \cdots\oplus \gg_{0}$  
and $\langle\cdot ,
\cdot \rangle_{\gg^{*}}$ the induced metric on $\gg^{*}$. Define a (Euclidean, adapted to the gradation) metric $\langle
\cdot ,\cdot \rangle$ on $\gh = t^{*}(\gg )$, which on $\gg$ and $\gg^{*}$
coincides with $\langle\cdot , \cdot \rangle_{\gg}$ and
$\langle\cdot , \cdot \rangle_{\gg^{*}}$ respectively and such
that $\gg$ is orthogonal to $\gg^{*}$. 
In this section we find obstructions for $\langle \cdot , \cdot\rangle $ to be admissible. 
Let
$$
B (X +\xi , Y +\eta ) = \xi (Y) +\eta (X), \quad X +\xi , Y + \eta
\in \gh = \gg \oplus \gg^{*}
$$
be the standard $\mathrm{Ad}_{H}$-invariant  metric of neutral
signature of $\gh$. Remark that
\begin{equation}\label{relatie}
\langle X +\xi , Y +\eta\rangle = B( X+\xi , \theta (Y +\eta
)),\quad X +\xi , Y+\eta \in \gh ,
\end{equation}
where $\theta : \gh \ra \gh$, with $\theta (X ) = X^{\sharp}$,
$\theta (\alpha ) = \alpha^{\flat}$ ($X\in \gg$, $\alpha \in
\gg^{*}$) is the involution defined by the Riemannian duality
induced by $\langle \cdot , \cdot \rangle_{\gg}$. We notice that
$\theta$ is $B$-orthogonal and $\theta (\gh_{i}) = \gh_{-i}$, for
any $-k\leq i\leq k$. Therefore, the metric $\langle \cdot , \cdot\rangle$ looks formally the same
as the standard admissible metric (\ref{angle}) on semisimple Lie algebras.

\bp\label{criterion-adm} The metric $\langle \cdot , \cdot \rangle$ is admissible if and only if 
\begin{equation}\label{theta}
\theta ( [ X+\xi , \theta  ([ Y , Z +\alpha ])  ] ) = [ \theta ([
X , \theta (Y) ]), Z +\alpha ] + [ Y, \theta  ([ X +\xi ,
\theta ( Z +\alpha )] ) ],
\end{equation}
for any $X +\xi \in \gg_{0} \oplus \gg^{*}$,
$Y\in \gg_{-}$ and $ Z+\alpha \in \gg\oplus \gg^{*}$. 
\ep 

\begin{proof}
Using (\ref{relatie}) and the properties of $\langle \cdot , \cdot
\rangle$, $B$ and $\theta$ stated above, we obtain (like in Corollary
\ref{check-ss}) $\mathrm{Ad}_{g}^{*} = \theta \mathrm{Ad}_{g^{-1}}
\theta$ and $(\mathrm{Ad}^{-}_{g})^{*} = (\theta
\mathrm{Ad}_{g^{-1}} \theta )\vert_{\gh_{-}}$, for any $g\in Q$.
Relations
(\ref{ad-star}) and (\ref{ad-star-add}) are equivalent to
the single condition:
$$
(\theta \mathrm{Ad}_{g}\theta )[ Y, Z +\alpha ] = [ (\theta
\mathrm{Ad}_{g} \theta ) (Y), (\theta \mathrm{Ad}_{g}\theta ) (Z
+\alpha ) ], \quad  \forall g\in Q,\ Y\in \gg_{-},\ Z +\alpha \in
\gh ,
$$
which is equivalent to (\ref{theta}), as required (recall that $\mathrm{Lie}(Q)
= \gh_{0} \oplus \gh_{+} = \gg_{0} \oplus \gg^{*}$).
\end{proof}

As a consequence, we obtain obstructions for $\langle \cdot , \cdot\rangle$ to be admissible.

\bc\label{consecinta} Suppose $\langle\cdot , \cdot \rangle$ is admissible. 
Then $\gg_{0}$ is abelian and $\langle\cdot , \cdot \rangle_{\gg}$ is
not $\mathrm{Ad}_{G_{0}}$-invariant.
\ec

\begin{proof}
Letting in (\ref{theta})  $X=0$, $Y, Z\in \gg_{-1}$ and $\alpha , \xi \in (\gg_{0})^{*}$, we obtain 
$[Y, \theta L_{X_{0}}(\xi ) ] =0$, where $X_{0}:= \theta (\alpha )\in \gg_{0}.$  Since the action of $\gg_{0}$
on $\gg_{-1}$ is exact and $\theta L_{X_{0}} (\xi )\in \gg_{0}$,  
it follows that  $L_{X_{0}}(\xi )=0$, i.e. $\gg_{0}$ is abelian,  as required. 
This proves the first statement.
For the second statement, we assume, by absurd, that $\langle\cdot , \cdot \rangle$ is admissible
and   $\langle\cdot , \cdot \rangle_{\gg}$ is $\mathrm{Ad}_{G_{0}}$-invariant.
We consider again relation (\ref{theta}), with $X=0.$ 
We take the $B$-inner product of this relation with $\beta \in
\gg^{*}$ and we obtain (as can be easily checked) 
\begin{equation}\label{needed-ctg}
 \langle \xi \circ
\mathrm{ad}_{\theta (\alpha ) }, \beta \circ \mathrm{ad}_{Y}
\rangle_{\gg^{*}}=\langle \xi \circ \mathrm{ad}_{\theta (\beta ) }, \alpha \circ
\mathrm{ad}_{Y} \rangle_{\gg^{*}}  , \quad \forall Y\in \gg_{-},\ \alpha , \beta,
\xi \in \gg^{*}.
\end{equation}
Let $0< j \leq k$ be fixed, $Y,Z, V \in \gg_{-j}$,  $\alpha := V^{\sharp}$, 
$\beta := Z^{\sharp}$, $\xi \in (\gg_{-j})^{*}$. 
Using the 
$\mathrm{Ad}_{G_{0}}$-invariance of $\langle\cdot , \cdot
\rangle_{\gg} $, 
\begin{equation}\label{aaa}
 \langle \xi\circ\mathrm{ad}_{\theta (\alpha ) } ,
\beta \circ\mathrm{ad}_{Y} \rangle_{\gg^{*}}  = \langle \xi \circ
\mathrm{ad}_{V}, \beta \circ \mathrm{ad}_{Y} \rangle_{\gg^{*}}  
= \sum_{i} \xi ( [ V, X_{i} ] ) \langle [ X_{i}, Z], Y
\rangle_{\gg},
\end{equation}
where $\{ X_{i} \}$ is an orthonormal basis of $\gg_{0}$.
From (\ref{aaa}), relation (\ref{needed-ctg}) becomes  
\begin{equation}\label{baze}
\sum_{i} [ Z, X_{i} ] \otimes [ V, X_{i} ] = \sum_{i} [ V, X_{i}]
\otimes [Z, X_{i} ],\quad Z,V\in\gg_{-j}. 
\end{equation}
We now show that (\ref{baze}) leads to a contradiction.
More precisely,  
the (exact) action of the abelian Lie algebra $\gg_{0}$ on $\gg_{-1}$ is
by skew-symmetric endomorphisms 
(because $\langle\cdot , \cdot\rangle_{\gg}$ is $\mathrm{Ad}_{G_{0}}$-invariant)
and it decomposes
into $1$- and $2$-dimensional irreducible representations. The
$1$-dimensional representations are trivial. The $2$-dimensional
ones are of the form (by choosing an orthonormal base in
$\gg_{-1}$)
$$
\gg_{0} \ni X \rightarrow\left(
\begin{tabular}{cc}
$0$ & $- \lambda (X)$\\
$\lambda (X)$ & $0$\\
\end{tabular}\right)
$$
where $\lambda (X) \in \mathbb{R}.$  Therefore, we can find $Z, V
\in \gg_{-1}$ (linearly independent), such that
$$
[ X, Z ] =  \lambda (X) V, \quad [X, V] = -\lambda (X) Z, \quad
\forall X\in \gg_{0}
$$
and $\lambda \in (\gg_{0})^{*}$ is non-trivial. We obtain
$\sum_{i} [ Z, X_{i} ] \otimes [V, X_{i} ] = - \sum_{i} \lambda
(X_{i})^{2} Z \otimes V$, which is non-symmetric in $V$, $Z$. This
contradicts relation (\ref{baze}). 

\end{proof}

Dmitri V. Alekseevsky: {\it Institute for Information Transmission
Problems}, B. Karetny per. 19, 101447, Moscow (Russia) and {\it
Masaryk University}, Kotlarska 2, 61137, Brno (Czech Republic),
dalekseevsky@iitp.ru\\

Liana David: {\it Institute of Mathematics 'Simion Stoilow' of the
Romanian Academy}, Research Unit 4, Calea Grivitei nr 21, Sector
1, Bucharest (Romania); liana.david@imar.ro

\end{document}